\documentclass[10pt, a4paper]{article} 

\usepackage[utf8]{inputenc} 

\usepackage{graphicx} 
\usepackage[english]{babel}
\usepackage{xcolor}

\usepackage{amsfonts,amsmath,amssymb,bbm,amsthm,amsbsy}
\usepackage[all]{xy}
\SelectTips{cm}{}

\usepackage[pagebackref,
    ,pdfborder={0 0 0}
  ,urlcolor=black,a4paper,hypertexnames=false]{hyperref}

\usepackage{color}
\usepackage{pdfcolmk}

\makeatletter
\def\blfootnote{\xdef\@thefnmark{}\@footnotetext}
\makeatother
\date{\today%
    \protect\blfootnote{\copyright{\ C.~L\"oh 2022}. 
    This work was supported by the CRC~1085 \emph{Higher Invariants} 
    (Universit\"at Regensburg, funded by the DFG).
    \\
    MSC~2010 classification: 57N65, 53C23, 20F65}}

\usepackage[nouppercase]{scrpage2}
\renewcommand{\headfont}{\sffamily}  
\renewcommand{\pnumfont}{\sffamily}
\automark[subsection]{section}
\usepackage{ifthen}
\deftripstyle{myheadings}
             {\ifthenelse{\isodd{\value{page}}}{\leftmark}{\pagemark}}
             {}
             {\ifthenelse{\isodd{\value{page}}}{\pagemark}{\rightmark}}%
             {}{}{}%
\def\args{\;\cdot\;}

\newtheorem{thm}{Theorem}[section]
\newtheorem{lemma}[thm]{Lemma}
\newtheorem{prop}[thm]{Proposition}
\newtheorem{cor}[thm]{Corollary}

\theoremstyle{remark}
\newtheorem{rem}[thm]{Remark}

\theoremstyle{definition}
\newtheorem{defi}[thm]{Definition}
\newtheorem{example}[thm]{Example}

\newtheorem{question}[thm]{Question}

\newcommand{\R}{\mathbb{R}}

\newcommand{\Q}{\mathbb{Q}}
\newcommand{\Z}{\mathbb{Z}}
\newcommand{\N}{\mathbb{N}}

\DeclareMathOperator{\SV}{SV}
\DeclareMathOperator{\id}{id}

\DeclareMathOperator{\Hom}{Hom}

\DeclareMathOperator{\vol}{vol}

\DeclareMathOperator{\Mfd}{Mfd}
\DeclareMathOperator{\Gap}{Gap}
\DeclareMathOperator{\minvol}{minvol}
\DeclareMathOperator{\MGap}{\Gap_{E}}
\DeclareMathOperator{\minvolent}{minvolent}
\DeclareMathOperator{\MVEGap}{\Gap_{\minvolent}}
\DeclareMathOperator{\Riemc}{Riem_1}

\DeclareMathOperator{\FHtype}{FH}
\def\FH#1{\FHtype_{#1}}
\DeclareMathOperator{\comp}{comp}

\def\SVG#1#2{\SV_{#2}(#1)}

\def\qand{\quad\text{and}\quad}

\def\fa#1{%
  \forall_{#1}\;\;\;}

\title{The spectrum of simplicial volume\\ with fixed fundamental group}
\author{Clara L\"oh}

\begin{document}

\maketitle

\begin{abstract}
  We study the spectrum of simplicial volume for closed manifolds with
  fixed fundamental group and relate the gap problem to rationality
  questions in bounded (co)homology. In particular, we show that in
  many cases this spectrum has a gap at zero. For such groups, this
  leads to corresponding gap results for the minimal volume entropy
  semi-norm and for the minimal volume entropy in dimension~$4$.
\end{abstract}

\section{Introduction}

The simplicial volume of an oriented closed connected manifold is the
$\ell^1$-semi-norm of its $\R$-fundamental class~\cite{vbc}
(Section~\ref{subsec:sv}). The simplicial volume is connected to
amenability, negative curvature, and Riemannian volume
estimates~\cite{vbc}.

\begin{defi}[spectrum of simplicial volume]
  Let $d\in \N$ and let $\Mfd_d$ denote the class of all oriented
  closed connected $d$-manifolds. 
  The \emph{spectrum of simplicial volume in dimension~$d$}
  is the set
  \[ \SV(d) := \bigl\{ \|M\| \bigm| M \in \Mfd_d \bigr\} \subset \R_{\geq 0}.
  \]
  Given a group~$\Gamma$, we write
  \[ \SVG d \Gamma := \bigl\{ \|M\| \bigm| M \in \Mfd_d,\ \pi_1(M) \cong \Gamma \bigr\}
  \subset \R_{\geq 0}.
  \]
\end{defi}

A subset~$V \subset \R_{\geq 0}$ \emph{has a gap at~$0$} if
there exists a~$c\in \R_{>0}$ with~$V \cap (0,c) = \emptyset$. 
The sets~$\SV(d)$ are known not to have a gap at
zero whenever~$d \geq 4$ (Section~\ref{subsec:spectrumsv}).
However, the problem is open for the spectrum with 
fixed fundamental group:

\begin{question}[gap problem with fixed fundamental group]\label{q:fixedpi1}
  Let $d \in \N$ and let $\Gamma$ be a finitely presented group
  with~$\dim_\R H_d(\Gamma;\R) < \infty$. Does the set~$\SVG d \Gamma$
  have a gap at zero?
\end{question}

Fundamental groups of closed manifolds are finitely presented. In the
context of homological properties of groups, it is reasonable to
further restrict the class of groups: We say that a group~$\Gamma$ has
\emph{type~$\FH d$} if it is finitely presented and satisfies~$\dim_\R
H_d(\Gamma;\R) < \infty$.

In the present article, we give partial positive answers to
Question~\ref{q:fixedpi1} and put this problem into a geometric
context.

\subsection{The spectrum of simplicial volume}\label{subsec:spectrumsv}

We first recall known results on the spectrum of simplicial volume.
On the one hand, we have generic structural results:

\begin{thm}[general structure~\protect{\cite[Remark~2.3]{heuerloeh_4mfd}%
      \cite[Theorem~B/E]{heuerloeh_trans}}]
  Let $d\in \N$.
  \begin{enumerate}
  \item The set~$\SV(d)$ is countable and closed under addition.
  \item The set $\SV(d)$ is contained in the set of right-computable
    real numbers; in particular, if $A \subset \N$ is
    a subset that is recursively enumerable but not recursive, then~$\sum_{n \in
      A}2^{-n}$ is \emph{not} in~$\SV(d)$.
  \end{enumerate}
\end{thm}

On the other hand, classification results in low dimensions
and stable commutator length, respectively, can be used to exhibit
concrete real numbers as simplicial volumes:

\begin{thm}[(no) gap~\protect{\cite[Example~2.4/2.5, Theorem~A]{heuerloeh_4mfd}}]
  \hfil
  \begin{enumerate}
  \item The sets~$\SV(0), \dots, \SV(3)$ have a gap at zero.
  \item If $d \in \N_{\geq 4}$, then $\SV(d)$ is dense in~$\R_{\geq
    0}$.
  \end{enumerate}  
\end{thm}

The most specific information is available in dimension~$4$:

\begin{thm}[dimension~$4$]\label{thm:dim4}
 The set~$\SV(4)$ contains
 \begin{itemize}
 \item all non-negative rationals~\cite[Theorem~B]{heuerloeh_4mfd};
 \item a dense set of transcendental numbers that
      is linearly independent over the field of algebraic
      numbers~\cite[Theorem~A, Theorem~C]{heuerloeh_trans};
 \item certain irrational algebraic numbers~\cite[Theorem~1.10]{ffflodha}.
 \end{itemize}
\end{thm}

The constructions from Theorem~\ref{thm:dim4} can be performed with
fundamental groups with a bounded number of generators and
relations~\cite[Section~8.4]{heuerloeh_4mfd}, but it is not clear from
the constructions whether it is possible to fix the group.

In contrast to the closed case, the spectrum of the (locally
finite) simplicial volume of oriented connected not necessarily
compact manifolds without boundary in dimensions~$\geq 4$
coincides with~$\R_{\geq 0} \cup \{\infty\}$~\cite{heuerloeh_lf}.

\subsection{Gaps and rationality}

We show that the gap behaviour of a given fundamental group is driven
by the rationality properties of the zero-norm subspace of singular
homology.

\begin{defi}
  Let $d \in \N$ and let $X$ be a topological space or a group.
  \begin{itemize}
  \item
    Then we
    write
    \begin{align*}
      N_d(X;\R)
      &:= \bigl\{ \alpha \in H_d(X;\R) \bigm| \|\alpha\|_1 = 0 \bigr\}
      \subset H_d(X;\R),
      \\
      B^d(X;\R)
      & := \bigl\{ \varphi \in H^d(X;\R) \bigm| \text{$\varphi$ is bounded} \bigr\}
      \subset H^d(X;\R).
    \end{align*}
  \item
    A subspace~$V \subset H_d(X;\R)$ is \emph{rational} if $V \cap H_d(X;\Q)$
    generates~$V$ over~$\R$.
    A subspace $V \subset H^d(X;\R)$ is \emph{rational} if $V \cap H^d(X;\Q)$
    generates~$V$ over~$\R$.
  \end{itemize}
\end{defi}
  
\begin{thm}[Section~\ref{sec:ratchar}]\label{thm:ratchar}
  Let $d\in \N_{\geq 4}$ and let $\Gamma$ be a group of type~$\FH d$.
  Then the following are equivalent:
  \begin{enumerate}
  \item\label{i:gap}
    The set $\SVG d \Gamma$ has a gap at zero.
  \item\label{i:groupgap}
    The set
   $ \bigl\{ \|\alpha\|_1
       \bigm| \text{$\alpha \in H_d(\Gamma;\R)$ is integral} \bigr\}
   $
    has a gap at zero.   
  \item\label{i:ratzero}
    The subspace~$N_d(\Gamma;\R)$ is rational in~$H_d(\Gamma;\R)$.
  \item\label{i:ratbounded}
    The subspace~$B^d(\Gamma;\R)$ is rational in~$H^d(\Gamma;\R)$.
  \end{enumerate}
\end{thm}

The proof in Section~\ref{sec:ratchar} shows that 
the implication~\ref{i:groupgap}. $\Longrightarrow$ \ref{i:gap}.  as
well as the equivalence of the properties \ref{i:groupgap}.,
\ref{i:ratzero}., and \ref{i:ratbounded}. also hold for~$d \in
\{0,1,2,3\}$. 


The rationality property~\ref{i:ratbounded}. is related to
a problem of Frigerio and Sisto in the context of quasi-isometrically
trivial extensions~\cite[Question~16]{frigeriosisto}. 

\subsection{Examples}\label{subsec:introexas}

The characterisation in Theorem~\ref{thm:ratchar} allows us to
establish that many groups admit a positive answer to
Question~\ref{q:fixedpi1}.  Let $d\in \N$. We write~$\Gap(d)$ for the
class of all groups~$\Gamma$ of type~$\FH d$ such that $\SVG d \Gamma$
has a gap at zero.

If~$N_d(\args;\R)$ is trivial or the full homology, then
$N_d(\args;\R)$ is rational in~$H_d(\args;\R)$ (and similarly for
cohomology). Therefore, we obtain:

\begin{example}[base cases] 
  Let $d\in \N_{\geq 4}$. The class~$\Gap(d)$ contains the following
  groups: 
  \begin{itemize}
  \item all amenable groups of type~$\FH d$ because they have trivial
    bounded cohomology~\cite{vbc,ivanov} (and thus~$\SVG d \Gamma =
    \{0\}$~\cite{vbc});
  \item more generally, all boundedly acyclic groups of type~$\FH d$;
    this includes the Thompson group~$F$~\cite{monod_F};
  \item all hyperbolic groups because they are of finite type
    and the $\ell^1$-semi-norm is a norm by the duality
    principle and Mineyev's results~\cite{mineyev};
  \item all finitely presented groups with~$\dim_\R H_d(\Gamma;\R) \leq 1$;
  \item all groups~$\Gamma$ of type~$\FH d$ whose comparison
    map~$H_b^d(\Gamma;\R) \longrightarrow H^d(\Gamma;\R)$ is trivial;
    this includes all groups of type~$\FH d$ whose classifying
    space admits an amenable open cover of multiplicity at
    most~$d$~\cite{vbc,ivanov,loehsauer_amenable}.
    Good bounds for such amenable multiplicities are, e.g., known for
    right-angled Artin groups~\cite{li_RAAG}. More generally,
    one can also consider multiplicities of (uniformly) boundedly acyclic
    open covers~\cite{ivanov_covers,liloehmoraschini}.
  \end{itemize}
\end{example}

\begin{example}[Thompson group~$T$]
  The Thompson group~$T$ lies in~$\Gap(d)$ for all~$d \in \N_{\geq 4}$:
  It is well-known that $T$ is finitely presented and has
  finite-dimensional cohomology in every degree~\cite{ghyssergiescu}.
  Moreover, $B^*(T;\R)$ is generated by the cup-powers of the
  Euler class~\cite{ffflmm_T,monodnariman}. Because the Euler class
  is rational, we see that $B^*(T;\R)$ is rational. We can thus
  apply Theorem~\ref{thm:ratchar} to conclude.
\end{example}

We have the following inheritance properties (proofs
are given in Section~\ref{sec:inherit}): 

\begin{example}[inheritance properties]
  For~$d \in \N_{\geq 4}$, we have:
  \begin{itemize}
  \item The class~$\Gap(d)$ is closed under taking (finite)
    free products.

    More generally, there is an inheritance principle
    for graphs of groups with amenable edge groups and
    vertex groups in~$\Gap(d)$ (Lemma~\ref{lem:graphs}).
  \item Let $\Gamma \in \bigcap_{k \in \{2,\dots, d\}} \Gap(k)$
    and $\Lambda \in \bigcap_{k \in \{2,\dots,d\}} \Gap(k)$. Then
    \[ \Gamma \times \Lambda \in \Gap(d).
    \]
  \item
    If $\Gamma$ is a group that contains a finite index
    subgroup in~$\Gap(d)$, then also~$\Gamma \in \Gap(d)$.
  \item
    Let $1 \longrightarrow A \longrightarrow \Gamma \longrightarrow
    \Lambda \longrightarrow 1$ be an extension of groups with
    boundedly acyclic (e.g., amenable) kernel~$A$. If~$\Lambda \in
    \Gap(d)$ and $\Gamma$ is of type~$\FH d$, then~$\Gamma \in
    \Gap(d)$.
  \item More generally: Let $f \colon \Gamma \longrightarrow \Lambda$
    be a group homomorphism that induces a surjection~$H^d_b(f;\R)
    \colon H^d_b(\Lambda;\R) \longrightarrow H^d_b(\Gamma;\R)$. If
    $\Lambda \in \Gap(d)$ and $\Gamma$ is of type~$\FH d$, then
    also~$\Gamma \in \Gap(d)$.
  \end{itemize}
\end{example}

However, it remains an open problem whether for all groups~$\Gamma$ of
type~$\FH d$ the space~$N_d(\Gamma;\R)$ is rational or not.

If we drop the finiteness conditions, then, in general, we cannot
expect a gap on integral classes:

\begin{example}\label{exa:infnogap}
  There exists a countable group~$\Gamma$ such that $\{\|\alpha\|_1
  \mid \alpha \in H_2(\Gamma;\R) \text{ is integral}\}$ has \emph{no}
  gap at zero: For each~$n \in \N_{>0}$, there exists a finitely
  presented group~$\Gamma_n$ with an integral class~$\alpha_n \in
  H_2(\Gamma_n;\R)$ satisfying
  \[ 0 < \| \alpha_n \|_1 < \frac 1n;
  \]
  for example, such groups and elements can be constructed via stable
  commutator length~\cite[Theorem~C]{heuerloeh_4mfd}.  Then the
  infinite free product~$\Gamma$ of the~$(\Gamma_n)_{n \in \N}$ has
  the claimed property. Clearly, this example~$\Gamma$ is not finitely
  generated and $\dim_\R H_2(\Gamma;\R) = \infty$.

  Taking products with fundamental groups of oriented closed connected
  hyperbolic manifolds and the standard cross-product estimates
  for~$\|\cdot\|_1$~\cite[Proposition~2.9]{heuerloeh_4mfd} show that
  such examples also exist in all degrees~$\geq 4$.
\end{example}

\subsection{Gap phenomena for geometric volumes}

In dimensions~$d\geq 4$, it does not seem to be known whether the
set of minimal volumes of all oriented closed connected smooth
$d$-manifolds has a gap at~$0$ or not. For a smooth manifold~$M$, the
\emph{minimal volume} is defined by
\[ \minvol (M) := \inf \bigl\{ \vol(M,g) \bigm| g \in \Riemc(M) \bigr\}, 
\]
where $\Riemc(M)$ denotes the set of all complete Riemannian metrics
on~$M$ whose sectional curvature lies everywhere in~$[-1,1]$.  The
following connections with the simplicial volume are
classical~\cite[Section~0.5]{vbc}:
\begin{itemize}
\item \emph{Main inequality.}
  For all oriented closed connected smooth $d$-manifolds~$M$,
  we have
  \[ \|M \| \leq (d-1)^d \cdot d! \cdot \minvol(M).
  \]
\item \emph{Isolation theorem.}
  For each~$d \in \N$, there exists a constant~$\varepsilon_d \in \R_{>0}$
  with the following isolation property: If $M$ is an oriented closed
  connected smooth $d$-manifold with~$\minvol(M) < \varepsilon_d$, then
  $\|M\| = 0$.
\end{itemize}
It is not known whether the vanishing of simplicial volume
implies the vanishing of the minimal volume. Therefore, the gap results
from Section~\ref{subsec:introexas} do not directly give gap results
for the minimal volume with fixed fundamental group.

Similarly, the corresponding gap problem for the minimal volume
entropy is open. For~$d\in \N$, we write~$\MVEGap(d)$ for the class of
all groups~$\Gamma$ of type~$\FH d$ such that the set of minimal
volume entropies~$\minvolent(M)$ of oriented closed connected smooth
$d$-manifolds~$M$ with fundamental group isomorphic to~$\Gamma$ has a
gap at~$0$.  In dimension~$4$, gaps for simplicial volume lead to gaps
for minimal volume entropy:

\begin{cor}[minimal volume entropy gaps in dimension~$4$]
  \hfil
  \begin{enumerate}
  \item We have~$\Gap(4) \subset \MVEGap(4)$.
  \item In particular, all the examples of groups in~$\Gap(4)$ listed
    in Section~\ref{subsec:introexas} lie in~$\MVEGap(4)$.
  \end{enumerate}
\end{cor}
\begin{proof}
  The second part is clear. For the first part, on the one hand, we use
  that the minimal volume entropy is a linear upper bound for the
  simplicial volume~\cite{bcg}; on the other hand, in dimension~$4$,
  the vanishing of simplicial volume implies the vanishing of the
  minimal entropy~\cite[Theorem~A]{suarezserrato} and whence of the
  minimal volume entropy~\cite{bcg}.
\end{proof}  

The volume entropy semi-norm~$\|\cdot\|_E$ is equivalent to the
$\ell^1$-semi-norm on singular homology~\cite{babenkosabourau_sn}.
Let~$\MGap(d)$ be the class of all groups~$\Gamma$ of type~$\FH d$
such that the set of volume entropy semi-norms~$\|[M]_\R\|_E$ of
oriented closed connected smooth $d$-manifolds~$M$ with fundamental
group isomorphic to~$\Gamma$ has a gap at~$0$.

\begin{cor}[volume entropy semi-norm gaps]
  Let $d\in \N$.
  \begin{enumerate}
  \item We have~$\Gap(d) \subset \MGap(d)$. 
  \item In particular, all the examples of groups in~$\Gap(d)$ listed
    in Section~\ref{subsec:introexas} lie in~$\MGap(d)$.
  \end{enumerate}
\end{cor}
\begin{proof}
  The first part follows from  the fact that $\|\cdot\|_E$
  and~$\|\cdot\|_1$ are equivalent on
  singular homology~\cite[Theorem~1.3]{babenkosabourau_sn},
  whence on fundamental classes of smooth manifolds.
  The second part is clear.
\end{proof}

The smooth Yamabe invariant can be viewed as a curvature integral
sibling of the minimal volume, defined in terms of scalar curvature
instead of sectional/Riemannian curvature.  If $d \in \N_{\geq 5}$ and
$\Gamma$ is of type~$\FH d$, then it is known that the truncated
smooth Yamabe invariant on oriented closed connected smooth spin
$d$-manifolds with fundamental group isomorphic to~$\Gamma$ has a gap
at~$0$; this is implicitly contained in the surgery inheritance
results for this version of the Yamabe
invariant~\cite[Section~1.4]{ammanndahlhumbert}.

\subsection*{Organisation of this article}

Basic notions are recalled in Section~\ref{sec:prelim}.  In
Section~\ref{sec:ratchar}, we prove Theorem~\ref{thm:ratchar}.
Finally, Section~\ref{sec:inherit} treats the inheritance properties
listed in Section~\ref{subsec:introexas}.

\subsection*{Acknowledgements}

I would like to thank Bernd Ammann and Francesco Fournier-Facio for
interesting discussions on related topics and the anonymous referee
for carefully reading the manuscript.

\section{Preliminaries}\label{sec:prelim}

We collect basic terminology and properties on simplicial volume
and bounded cohomology~\cite{vbc}. 

\subsection{The $\ell^1$-semi-norm and simplicial volume}\label{subsec:sv}

\begin{defi}[$\ell^1$-semi-norm]\label{def:l1}
  Let $X$ be a space or a group and let $d \in \N$. For~$\alpha \in H_d(X;\R)$,
  we set
  \[ \|\alpha\|_1
  := \inf \bigl\{ |c|_1 \bigm| \text{$c \in C_d(X;\R)$, $\partial c=0$, $[c] = \alpha$} \bigr\}
  \in \R_{\geq 0}.
  \]
  Here, $C_*(X;\R)$ denotes the singular chain complex if $X$ is a
  space; if $X$ is a group, $C_*(X;\R)$ can be taken to be the chain
  complex of the simplicial resolution or the singular chain complex
  of a classifying space~$B\Gamma$ (these chain complexes are
  boundedly chain homotopy equivalent with respect to~$|\cdot|_1$).
  Moreover, $|\cdot|_1$ denotes the $\ell^1$-norm on~$C_*(X;\R)$ with
  respect to the basis given by all singular simplices (or all
  simplicial tuples, respectively).
\end{defi}

The $\ell^1$-semi-norm on~$H_*(\args;\R)$ is functorial in the following
sense: If $f \colon X \longrightarrow Y$ is a continuous map (or group
homomorphism, respectively) and $\alpha \in H_d(X;\R)$, then
\[ \bigl\| H_d(f;\R) (\alpha)\bigr\|_1 \leq \|\alpha\|_1.
\]

\begin{defi}[simplicial volume~\cite{munkholm,vbc}]
  The \emph{simplicial volume} of an oriented closed connected
  $d$-manifold~$M$ is defined as
  \[ \|M\| := \bigl\| [M]_\R \bigr\|_1,
  \]
  where $[M]_\R \in H_d(M;\R)$ denotes the $\R$-fundamental class
  of~$M$.
\end{defi}

\subsection{Bounded cohomology and duality}


The bounded cohomology of groups or spaces is $H_b^*(\args;\R) :=
H^*\bigl( C_*(\args;\R)^{\#}\bigr), $ where $C_*(\args;\R)^{\#}$
denotes the topological dual with respect to~$|\cdot|_1$
(the latter is introduced in Definition~\ref{def:l1}). 
Forgetting boundedness induces a natural transformation~$\comp^*
\colon H^*_b(\args;\R) \Longrightarrow H^*(\args;\R)$, the
\emph{comparison map}. Classes in the image of the comparison map are
called \emph{bounded}. Evaluating cocycles on cycles induces a
Kronecker product~$\langle\args,\! \args\rangle$, which is compatible
with the comparison map.

\begin{prop}[duality principle~\protect{\cite[p.~16]{vbc}}]\label{prop:dualityprinc}
  Let $d\in \N$, let $X$ be a space/group, and let $\alpha \in
  H_d(X;\R)$.  Then
  \[ \|\alpha\|_1
  = \sup \Bigl \{\frac1{\|\varphi\|_\infty}
  \Bigm| \varphi \in H^d_b(X;\R),\ \langle \varphi,\alpha \rangle = 1
  \Bigr\}.
  \]
\end{prop}

We will also use the following version of the duality principle:

\begin{cor}\label{cor:annihilator} 
  Let $d\in \N$, let $X$ be a space/group with~$\dim_\R H_d(X;\R) <
  \infty$.  Then
  \[ B^d(X;\R)
  = \bigl\{ \varphi \in H^d(X;\R)
  \bigm| \fa{\alpha \in N_d(X;\R)}
  \langle \varphi, \alpha \rangle = 0
  \bigr\}.
  \]
\end{cor}
\begin{proof}
  By the duality principle (Proposition~\ref{prop:dualityprinc}),
  we have
  \begin{align*}
    N_d(X;\R)
  & = \bigl\{ \alpha \in H_d(X;\R)
  \bigm| \fa{\varphi \in H_b^d(X;\R)} \langle \varphi, \alpha\rangle = 0
  \bigr\}
  \\
  & = \bigl\{ \alpha \in H_d(X;\R)
  \bigm| \fa{\varphi \in B^d(X;\R)} \langle \varphi, \alpha\rangle = 0
  \bigr\}.
  \end{align*}
  Because $H_d(X;\R)$ is finite-dimensional and $H^d(X;\R) \cong_\R
  \Hom_\R(H_d(X;\R),\R)$ via the evaluation map, the annihilator
  \[ \bigl\{ \varphi \in H^d(X;\R)
  \bigm| \fa{\alpha \in N_d(X;\R)}
  \langle \varphi, \alpha\rangle =0
  \bigr\}
  \]
  of this null space coincides with~$B^d(X;\R)$.
\end{proof}

\subsection{Normed Thom realisation}

Classical Thom realisation and surgery allow us to construct manifolds
from group homology classes with controlled simplicial volume:

\begin{thm}[\protect{\cite[(proof of) Theorem~8.1]{heuerloeh_4mfd}}]
  \label{thm:normedthom}
  Let $d \in \N_{\geq 4}$. Then, there exists a constant~$K_d \in \N_{>0}$
  with the following property: If $\Gamma$ is a finitely presented group
  and $\alpha \in H_d(\Gamma;\R)$ is an integral class, then there exists
  an oriented closed connected $d$-manifold~$M$ with~$\pi_1(M) \cong \Gamma$
  and a~$K\in \{1,\dots, K_d\}$ such that
  \[ \| M\| = K \cdot \|\alpha\|_1.
  \]
\end{thm}

\section{Gaps via rationality}\label{sec:ratchar}

In this section, we prove Theorem~\ref{thm:ratchar}.
More precisely, we show:
\begin{itemize}
\item the equivalence \ref{i:gap}. $\Longleftrightarrow$
  \ref{i:groupgap}.  in Section~\ref{subsec:intlattice} via the
  mapping theorem and normed Thom realisation;
\item the equivalence \ref{i:groupgap}. $\Longleftrightarrow$
  \ref{i:ratzero}.  in Section~\ref{subsec:ratzero} through basic
  properties of integer lattices in vector spaces;
\item the equivalence \ref{i:ratzero}. $\Longleftrightarrow$
  \ref{i:ratbounded}.  in Section~\ref{subsec:ratbounded} by the
  duality principle;
\end{itemize}

\subsection{The integral lattice}\label{subsec:intlattice}

Let $X$ be a space or a group.
A class in~$H_d(X;\R)$ is called \emph{integral} if it is in the
image of the change of coefficients map~$H_d(X;\Z)
\longrightarrow H_d(X;\R)$. We write
\[ Z_d(X)
:= \bigl\{ \alpha \in H_d(X;\R)
\bigm| \text{$\alpha$ is integral}
\bigr\}
\]
for the $\Z$-submodule of~$H_d(X;\R)$ of integral classes.
Normed Thom realisation shows that $\SVG d \Gamma$ is roughly
the same as~$\{ \|\alpha\|_1 \mid \alpha \in Z_d(\Gamma)\}$:

\begin{proof}[Proof of Theorem~\ref{thm:ratchar},
    \ref{i:groupgap}. $\Longrightarrow$ \ref{i:gap}]
  Let $M \in \Mfd_d$ satisfying~$\pi_1(M) \cong \Gamma$ and let $f
  \colon M \longrightarrow B\Gamma$ be the classifying map. As $f$
  induces an isomorphism on the level of fundamental groups, we obtain
  from the mapping theorem~\cite[Section~3.1]{vbc} and the duality
  principle (Proposition~\ref{prop:dualityprinc}) that
  \[ 
    \| M \|
  = \bigl\| [M]_\R \bigr\|_1
  = \bigl\| H_d(f;\R)([M]_\R)\bigr\|_1.
  \]
  Moreover, $[M]_\R \in H_d(M;\R)$ is an integral class 
  and so $H_d(f;\R)([M]_\R) \in Z_d(\Gamma)$.

  Hence, if $\|\cdot\|_1$ has a gap at zero on~$Z_d(\Gamma)$,
  then also~$\SVG d \Gamma$ has a gap at zero.
\end{proof}

\begin{proof}[Proof of Theorem~\ref{thm:ratchar},
    \ref{i:gap}. $\Longrightarrow$ \ref{i:groupgap}]
  Let $\SVG d \Gamma$ have a gap~$c$ at zero and let $K_d \in \N_{>0}$
  be a constant for normed Thom realisation in dimension~$d$
  (Theorem~\ref{thm:normedthom}). Then $c/K_d$ is a gap
  for~$\|\cdot\|_1$ on~$Z_d(\Gamma)$:
  
  Let $\alpha \in Z_d(\Gamma)$ with~$\|\alpha\|_1 \neq 0$. Normed Thom
  realisation shows that there exists an~$M \in \Mfd_d$ with~$\pi_1(M)
  \cong \Gamma$ and $\|M\| = K \cdot \|\alpha\|_1$ with~$K \in
  \{1,\dots,K_d\}$. In particular, we obtain~$\|\alpha\|_1 \geq
  \|M\|/K \geq c/K_d$, as claimed.
\end{proof}

\begin{rem}[lattices]\label{rem:intlattice}
  Let $V$ be a finite-dimensional $\R$-vector space. Then $V$ carries
  a canonical topology (induced by any Euclidean inner product
  on~$V$). A \emph{lattice} in~$V$ is a $\Z$-submodule that is
  discrete with respect to the canonical topology. We recall two basic
  facts on lattices:
  \begin{itemize}
  \item
    If $\|\cdot\|$ is a norm on~$V$ and $L \subset V$ is a
    lattice, then $\{ \|x\| \mid x \in L \setminus \{0\} \}$
    has a gap at zero.

    (The corresponding statement for \emph{semi-}norms is false, in
    general: The semi-norm $x \mapsto |x_1 - \sqrt 2 \cdot x_2|$
    on~$\R^2$ does not have a gap on the standard lattice~$\Z^2
    \subset \R^2$. Even worse, this semi-norm is non-degenerate
    on~$\Z^2$.)
  \item
    If $L \subset V$ is a cocompact lattice, then $V$ has
    an $\R$-basis consisting of elements of~$L$.
  \end{itemize}

  Our main example is: Let $d \in \N$ and let $X$ be a space/group
  satisfying $\dim_\R H_d(X;\R) < \infty$.  Then, by the universal
  coefficient theorem, $Z_d(X)$ is a lattice in~$H_d(X;\R)$.
\end{rem}

\subsection{Rationality of the zero-norm subspace}\label{subsec:ratzero}

In the following, we consider the quotient space~$Q_d(\Gamma;\R)
:= H_d(\Gamma;\R)/N_d(\Gamma;\R)$. By construction,  
the quotient semi-norm of~$\|\cdot\|_1$ on~$Q_d(\Gamma;\R)$
is a norm and the canonical projection~$\pi \colon H_d(\Gamma;\R)
\longrightarrow Q_d(\Gamma;\R)$ is isometric. We 
denote the quotient norm also by~$\|\cdot\|_1$. 

\begin{proof}[Proof of Theorem~\ref{thm:ratchar},
    \ref{i:ratzero}. $\Longrightarrow$ \ref{i:groupgap}]
  Let $N_d(\Gamma;\R)$ be rational in~$H_d(\Gamma;\R)$.  Because
  $N_d(\Gamma;\R)$ is rational and $Z_d(\Gamma)$ is a lattice
  in~$H_d(\Gamma;\R)$ (Remark~\ref{rem:intlattice}), the
  image~$\pi(Z_d(\Gamma))$ is a lattice in the finite-dimensional
  $\R$-vector
  space~$Q_d(\Gamma;\R)$~\cite[Corollary~10.3]{barvinok}. In
  particular, the norm~$\|\cdot\|_1$ has a gap at~$0$
  on~$\pi(Z_d(\Gamma))$ (Remark~\ref{rem:intlattice}).  Therefore,
  also
  \[ \bigl\{ \|\alpha\|_1 \bigm| \alpha \in Z_d(\Gamma) \bigr\}
  =  \bigl\{ \|\pi(\alpha)\|_1 \bigm| \alpha \in Z_d(\Gamma) \bigr\}
  = \bigl\{ \|\beta\|_1 \bigm| \beta \in \pi(Z_d(\Gamma)\bigr\}
  \]
  has a gap at zero.
\end{proof}

\begin{proof}[Proof of Theorem~\ref{thm:ratchar},
    \ref{i:groupgap}. $\Longrightarrow$ \ref{i:ratzero}]
  Let $\|\cdot\|_1$ have a gap~$c$ at zero on~$Z_d(\Gamma)$. 
  We show that $N_d(\Gamma;\R)$ is rational in~$H_d(\Gamma;\R)$:

  Because $Z_d(\Gamma)$ is a lattice in~$H_d(\Gamma;\R)$
  (Remark~\ref{rem:intlattice}), there exists a
  tuple~$(v_1,\dots,v_n)$ of elements of~$Z_d(\Gamma)$ that is an
  $\R$-basis for~$H_d(\Gamma;\R)$ (Remark~\ref{rem:intlattice}).  Let
  $\alpha \in N_d(\Gamma;\R)$. We write
  \[ \alpha = \sum_{j=1}^n \lambda_j \cdot v_j
  \]
  with~$\lambda_1,\dots, \lambda_n \in \R$. Given $N \in \N_{>0}$,
  simultaneous Dirichlet
  approximation~\cite[Theorem~II.1.A]{schmidt_diophantine} shows that
  there exist~$p_{N,1},\dots, p_{N,n} \in \Z$ and $q_N \in
  \{1,\dots,N\}$ with
  \[ \fa{j \in \{1,\dots,n\}}
   \Bigl|\lambda_j - \frac{p_{N,j}}{q_N} \Bigr| < \frac 1{q_N \cdot N^{1/n}}. 
  \]
  Then the class
  $\alpha_N
     := \sum_{j=1}^n p_{N,j} \cdot v_j
  $ 
  lies in~$Z_d(\Gamma)$ and
  \[ 
  \| q_N \cdot \alpha - \alpha_N \|_1
  \leq \sum_{j=1}^n |q_N \cdot \lambda_j - p_{N,j}| \cdot \|v_j\|_1
  \leq \sum_{j=1}^n \frac1{N^{1/n}} \cdot \|v_j\|_1.
  \]
  Because $q_N \cdot \alpha \in N_d(\Gamma;\R)$, we
  obtain
  $\|\alpha_N \|_1 = \|q_N \cdot \alpha - \alpha_N\|_1$
  and so the previous estimate and the gap~$c$ show that
  $\|\alpha_N\|_1 = 0$ for all large enough~$N$.
  Hence, $\alpha_N \in N_d(\Gamma;\R) \cap Z_d(\Gamma)$
  and $1/q_N \cdot \alpha_N \in N_d(\Gamma;\R) \cap H_d(\Gamma;\Q)$.

  We now consider the standard topology on the finite-dimensional
  $\R$-vector space~$H_d(\Gamma;\R)$. Then the choice of the
  approximating coefficients shows that $\lim_{N \to \infty} 1/q_N
  \cdot \alpha_N =\alpha$.

  In conclusion, $\alpha$ lies in the closure of $N_d(\Gamma;\R) \cap
  H_d(\Gamma;\Q)$ with respect to the standard topology. As
  $\R$-subspaces of finite-dimensional $\R$-vector spaces are closed
  in the standard topology, $\alpha$ lies in the $\R$-subspace
  generated by~$N_d(\Gamma;\R) \cap H_d(\Gamma;\Q)$.  This shows that
  $N_d(\Gamma;\R)$ indeed is rational.
\end{proof}

\subsection{Rationality of the bounded subspace}\label{subsec:ratbounded}

\begin{proof}[Proof of Theorem~\ref{thm:ratchar},
    \ref{i:ratzero}. $\Longleftrightarrow$ \ref{i:ratbounded}]
  This is a consequence of Corollary~\ref{cor:annihilator}: By linear
  algebra over~$\Q$, an $\R$-subspace of~$H_d(\Gamma;\R)$ is rational
  if and only if its annihliator is rational in the dual $\R$-vector
  space.  Thus, $N_d(\Gamma;\R)$ is rational in~$H_d(\Gamma;\R)$ if
  and only if $B^d(\Gamma;\R)$ is rational in the
  dual~$H^d(\Gamma;\R)$ of~$H_d(\Gamma;\R)$.
\end{proof}

\section{Inheritance properties}\label{sec:inherit}

We prove the inheritance properties from Section~\ref{subsec:introexas}. 

\begin{lemma}[free products]
  Let $d \in \N_{\geq 4}$. Then $\Gap(d)$ is closed under
  taking (finite) free products.
\end{lemma}
\begin{proof}
  Let $\Gamma, \Lambda \in \Gap(d)$. We show that $\Gamma * \Lambda
  \in \Gap(d)$:

  With $\Gamma$ and $\Lambda$ also~$\Gamma * \Lambda$ is of type~$\FH
  d$ (finitely presented groups are closed under free products and the
  homology is finite-dimensional by the Mayer--Vietoris sequence).  By
  Theorem~\ref{thm:ratchar}, we know that $N_d(\Gamma;\R)$ and
  $N_d(\Lambda;\R)$ are rational and it suffices to show that
  $N_d(\Gamma * \Lambda;\R)$ is rational:

  The inclusions/projections~$i,j$ and $p,q$, respectively, of the
  summands of the free product~$\Gamma * \Lambda$ induce the
  Mayer--Vietoris $\R$-isomorphism~$\varphi \colon H_d(\Gamma *
  \Lambda ;\R) \longrightarrow H_d(\Gamma;\R) \oplus
  H_d(\Lambda;\R)$. Under this isomorphism, $N_d(\Gamma * \Lambda;\R)$
  corresponds to~$N_d(\Gamma;\R) \oplus N_d(\Lambda;\R)$: If $\alpha
  \in N_d(\Gamma*\Lambda;\R)$, then
  \[ \bigl\| H_d(p;\R) (\alpha) \bigr\|_1 \leq \|\alpha\|_1 = 0
  \qand \bigl\| H_d(q;\R) (\alpha) \bigr\|_1 \leq \|\alpha\|_1 = 0
  \]
  and so $\varphi(\alpha) \in N_d(\Gamma;\R) \oplus N_d(\Lambda;\R)$. 
  Conversely, if $(\alpha,\beta) \in N_d(\Gamma;\R) \oplus N_d(\Lambda;\R)$,
  then
  \[ \bigl\| \varphi^{-1}(\alpha,\beta)
  \bigr\|_1
  = \bigl\| H_d(i;\R)(\alpha) + H_d(j;\R)(\beta) 
  \bigr\|_1
  \leq \|\alpha\|_1 + \|\beta\|_1
  \leq 0 
  \]
  and thus~$\varphi^{-1}(\alpha,\beta) \in N_d(\Gamma*\Lambda;\R)$. 

  Because $\varphi$ maps rational subspaces to rational subspaces,
  also $N_d(\Gamma*\Lambda;\R)$ is rational.
\end{proof}

\begin{lemma}[graphs of groups]\label{lem:graphs}
  Let $d \in \N_{\geq 4}$, let $G$ be a graph of groups on a finite
  graph~$(V,E)$, whose vertex groups~$(G_v)_{v\in V}$ lie in~$\Gap(d)$
  and whose edge groups~$(G_e)_{e\in E}$ are amenable. Let
  $\Gamma$ be the fundamental group of~$G$.  If $\Gamma$ is of
  type~$\FH d$, then also~$\Gamma \in \Gap(d)$.
\end{lemma}
\begin{proof}
  By Theorem~\ref{thm:ratchar}, $N_d(G_v;\R)$ is rational for
  all~$v\in V$ and it suffices to show that $N_d(\Gamma;\R)$ is
  rational.

  We consider the following commutative diagram:
  \[ \xymatrix{%
    H^d_b(\Gamma;\R) \ar[r]^-{F_b}
    \ar[d]_{\comp^d_\Gamma}
    & \bigoplus_{v\in V} H^d_b(G_v;\R)
    \ar[d]^{\bigoplus_{v \in V} \comp^d_{G_v}}
    \\
    H^d(\Gamma;\R) \ar[r]_-{F}
    & \bigoplus_{v \in V} H^d(G_v;\R)
    }
  \]
  Here, $F_b$ and $F$ denote the maps induced by the inclusions of the
  vertex groups on bounded cohomology and cohomology, respectively.
  The upper horizontal arrow~$F_b$ is surjective~\cite{bbfipp}.
  Hence, the diagram implies that
  \[ F \bigl(B^d(\Gamma;\R)\bigr)
     = \bigoplus_{v \in V} B^d(G_v;\R).
  \]
  The hypothesis that~$G_v \in \Gap(d)$ for all~$v\in V$ shows
  that the right-hand side is rational.
  Moreover, the map~$F$ is rational because it is induced by group
  homomorphisms; in particular, the kernel of~$F$ is rational.
  Therefore, also $B^d(\Gamma;\R)$ is rational.
\end{proof}

The statement of Lemma~\ref{lem:graphs} can be generalised to
uniformly boundedly acyclic edge groups by using the corresponding
result on bounded cohomology of such graphs of
groups~\cite[Theorem~8.11]{liloehmoraschini}.

In the situation of Lemma~\ref{lem:graphs}, we have the following
sufficient condition for the group~$\Gamma$ to be of type~$\FH d$:
By hypothesis, all vertex groups are of type~$\FH d$. If all edge
groups are of type~$\FH {d+1}$, then the Mayer--Vietoris sequence
in the proof of Lemma~\ref{lem:graphs} shows that $\Gamma$ is
of type~$\FH d$.

\begin{lemma}[products]
  Let $d \in \N_{\geq 4}$ and let $\Gamma \in \bigcap_{k \in \{2,\dots,d\}}\Gap(k)$
  and $\Lambda \in \bigcap_{k \in \{2,\dots,d\}} \Gap(k)$.
  Then $\Gamma \times \Lambda \in \Gap(d)$.
\end{lemma}
\begin{proof}
  As $\Gamma$ and $\Lambda$ are of type~$\FH d$, also~$\Gamma \times
  \Lambda$ is of type~$\FH d$ (finitely presented groups are closed
  under finite products; and the cohomological K\"unneth theorem).

  By Theorem~\ref{thm:ratchar}, we know that $N_k(\Gamma;\R)$ and
  $N_k(\Lambda;\R)$ are rational for all~$k \in \{2,\dots,d\}$
  and it suffices to show that $N_d(\Gamma\times \Lambda;\R)$
  is rational:

  More precisely, we show that, under the K\"unneth isomorphism,
  $N_d(\Gamma\times\Lambda;\R)$ corresponds to
  \[
  N := 
  \sum_{j=0}^d
  \bigl(N_j(\Gamma;\R) \otimes_\R H_{d-j}(\Lambda;\R)
  + H_j(\Gamma;\R) \otimes_\R N_{d-j}(\Lambda;\R) \bigr).
  \]
  Because the K\"unneth isomorphism preserves rational subspaces
  and because $N_0(\args;\R) = 0$ and $N_1(\args;\R) = H_1(\args;\R)$
  are always rational, this would finish the proof.

  The standard estimate for the homological cross-product
  (via the shuffle description of the Eilenberg--Zilber map) 
  shows that~$N \subset N_d(\Gamma \times \Lambda;\R)$.  In order to
  prove the converse inclusion~$N_d(\Gamma\times\Lambda;\R) \subset
  N$, we proceed as follows:

  We consider the bilinear form
  \[ \langle \args,\!\args\rangle
  \colon B^{\leq d}(\Gamma;\R) \times H_{\leq d}(\Gamma;\R)
  \longrightarrow \R.
  \]
  The description of the bounded part from
  Corollary~\ref{cor:annihilator} and elementary finite-dimensional
  linear algebra show that there exist families~$(\varphi_i)_{i\in I_1}$
  in~$B^{\leq d}(\Gamma;\R)$ and $(\alpha_i)_{i \in I_1 \sqcup I_0}$
  in~$H_{\leq d}(\Gamma;\R)$ with the following properties:
  \begin{itemize}
  \item The family~$(\alpha_i)_{i \in I_0}$ is an $\R$-basis
    of~$N_{\leq d}(\Gamma;\R)$.
  \item The family~$(\alpha_i)_{i \in I_0 \sqcup I_1}$ is an $\R$-basis
    of~$H_{\leq d}(\Gamma;\R)$.
  \item The family~$(\varphi_i)_{i \in I_1}$ is an $\R$-basis
    of~$B^{\leq d}(\Gamma;\R)$.
  \item For all~$i,j \in I_1$, we have
    \[ \langle \varphi_i, \alpha_j \rangle = \delta_{ij}.
    \]
  \end{itemize}
  Similarly, we obtain such families~$(\psi_j)_{j \in J_1}$ and
  $(\beta_j)_{j \in J_1 \sqcup J_0}$ for~$\Lambda$.

  Let $\alpha \in N_d(\Gamma\times\Lambda;\R)$. Using the
  K\"unneth isomorphism, we write
  (where $I := I_1 \sqcup I_0$ and $J := J_1 \sqcup J_0$) 
  \[ \alpha = \sum_{(i,j) \in I \times J} \lambda_{ij} \cdot \alpha_i \times \beta_j
  \]
  for suitable real coefficients~$\lambda_{ij}$.
  Let $(i_1,j_1) \in I_1 \times J_1$. Then $\lambda_{i_1,j_1} = 0$
  as the following computation shows:
  \begin{align*}
    |\lambda_{i_1,j_1}|
    & = \Bigl|
    \Bigl\langle
    \varphi_{i_1} \times \psi_{j_1},
    \sum_{(i,j) \in I \times J} \lambda_{ij} \cdot \alpha_i \times \beta_j
    \Bigr\rangle
    \Bigr|
    \\
    & = \bigl|\langle \varphi_{i_1} \times \psi_{j_1},
    \alpha \rangle \bigr|
    \\
    & \leq \|\varphi_{i_1}\|_\infty \cdot \|\psi_{j_1}\|_\infty
    \cdot \|\alpha\|_1
    \\
    & = 0
  \end{align*}
  Therefore, $\alpha \in N$.
\end{proof}

\begin{lemma}[finite index supergroups]
  Let $d\in \N_{\geq 4}$ and let $\Gamma$ be a group that contains a
  finite index subgroup~$\Lambda$ with~$\Lambda \in
  \Gap(d)$. Then~$\Gamma \in \Gap(d)$.
\end{lemma}
\begin{proof}
  By Theorem~\ref{thm:ratchar}, $N_d(\Lambda;\R)$ is rational and it
  suffices to show that $N_d(\Gamma;\R)$ is rational and that $\Gamma$
  has type~$\FH d$:

  Let $i \colon \Lambda \longrightarrow \Gamma$ denote the inclusion.
  Because $[\Gamma:\Lambda] < \infty$ and $[\Gamma:\Lambda]$ is a unit
  in~$\R$, there is a homological transfer map $t_d \colon
  H_d(\Gamma;\R) \longrightarrow H_d(\Lambda;\R)$, which satisfies
  \[ H_d(i;\R) \circ t_d = [\Gamma:\Lambda] \cdot \id_{H_d(\Gamma;\R)}.
  \]
  In particular, $\dim_\R H_d(\Gamma;\R) \leq \dim_\R H_d(\Lambda;\R)
  < \infty$.  Moreover, because $\Gamma$ contains a finitely presented
  subgroup of finite index (namely~$\Lambda$), also $\Gamma$ is
  finitely presented.

  We now show that
  $N_d(\Gamma;\R) = H_d(i;\R) \bigl(N_d(\Lambda;\R)\bigr)
  $: 
  Clearly, the right-hand side is contained in~$N_d(\Gamma;\R)$.
  Conversely, let $\alpha \in N_d(\Gamma;\R)$. We consider~$\widetilde
  \alpha := 1/[\Gamma:\Lambda] \cdot t_d(\alpha) \in H_d(\Lambda;\R)$.
  The explicit construction of the transfer~$t_d$ through lifts of singular
  simplices shows that
  \[ \|\widetilde \alpha\|_1
  \leq \frac1{[\Gamma:\Lambda]} \cdot [\Gamma:\Lambda] \cdot \|\alpha\|_1 = 0.
  \]
  Hence, $\widetilde \alpha \in N_d(\Lambda;\R)$. By construction,
  \[ \alpha = \frac1{[\Gamma:\Lambda]} \cdot H_d(i;\R) \bigl(t_d(\alpha)\bigr)
  = H_d(i;\R) (\widetilde \alpha).
  \]
  This proves the claimed description of~$N_d(\Gamma;\R)$. 

  Finally, because $H_d(i;\R)$ preserves rational subspaces, the
  rationality of the subspace~$N_d(\Lambda;\R)$ implies the rationality
  of~$N_d(\Gamma;\R)$.

  Alternatively, one could also use the cohomological transfer
  in (bounded) cohomology.
\end{proof}

\begin{lemma}[epis on bounded cohomology]\label{lem:bcepi}
  Let $d \in \N_{\geq 4}$, let $f \colon \Gamma \longrightarrow
  \Lambda$ be a group homomorphism that induces a
  surjection~$H^d_b(f;\R) \colon H^d_b(\Lambda;\R) \longrightarrow
  H^d_b(\Gamma;\R)$, let $\Lambda \in \Gap(d)$, and let $\Gamma$ be of
  type~$\FH d$.  Then~$\Gamma \in \Gap(d)$.
\end{lemma}
\begin{proof}
  By Theorem~\ref{thm:ratchar}, $B^d(\Lambda;\R)$ is rational
  and it suffices to show that $B^d(\Gamma;\R)$ is rational.
  The commutative diagram
  \[ \xymatrix{%
    H^d_b(\Lambda;\R) \ar[r]^-{H^d_b(f;\R)}
    \ar[d]_{\comp^d_\Lambda}
    & H^d_b(\Gamma;\R)
    \ar[d]^{\comp^d_\Gamma}
    \\
    H^d(\Lambda;\R) \ar[r]_-{H^d(f;\R)}
    & H^d(\Gamma;\R)
    }
  \]
  and the surjectivity of the upper arrow~$H^d_b(f;\R)$ imply 
  that
  \[ B^d(\Gamma;\R) = H^d(f;\R) \bigl( B^d(\Lambda;\R)\bigr).
  \]
  As $B^d(\Lambda;\R)$ is rational in~$H^d(\Lambda;\R)$ and as
  the induced homomorphism~$H^d(f;\R)$ preserves rationality,
  we obtain that also $B^d(\Gamma;\R)$ is rational.
\end{proof}

\begin{lemma}[boundedly acyclic extensions]
  Let $d \in \N_{\geq 4}$, let $1 \longrightarrow A \longrightarrow
  \Gamma \longrightarrow \Lambda \longrightarrow 1$ be an extension of
  groups with boundedly acyclic kernel~$A$, let $\Lambda \in \Gap(d)$, 
  and let $\Gamma$ be of type~$\FH d$. Then $\Gamma \in \Gap(d)$.
\end{lemma}
\begin{proof}
  Let $\pi \colon \Gamma \longrightarrow \Lambda$ be the epimorphism
  of the given short exact sequence. Because $\ker \pi \cong A$ is
  boundedly acyclic, the map~$H^d_b(\pi;\R) \colon H^d_b(\Lambda;\R)
  \longrightarrow H^d_b(\Gamma;\R)$ is an
  isomorphism; this can be seen from the Hochschild--Serre
  spectral sequence in bounded cohomology~\cite[Chapter~12]{monod}. 
  Therefore, Lemma~\ref{lem:bcepi} applies.
\end{proof}

\section*{Statements and declarations}

\emph{Funding.} 
This work was supported by the CRC~1085 \emph{Higher Invariants} (Universit\"at
Regensburg, funded by the DFG).

\noindent
\emph{Conflicts of interest.}
There are no conflicts of interest to declare.

\noindent
\emph{Author contribution.}
This is a single-author paper. The contribution is~$100\%$.

\noindent
\emph{Data availability statement.}
No external data is processed in this project.

{\small
\bibliographystyle{alpha}
\bibliography{bib_l1}}

\vfill

\noindent
\emph{Clara L\"oh}\\[.5em]
  {\small
  \begin{tabular}{@{\qquad}l}
    Fakult\"at f\"ur Mathematik,
    Universit\"at Regensburg,
    93040 Regensburg\\
    \textsf{clara.loeh@mathematik.uni-r.de}, 
    \textsf{https://loeh.app.ur.de}
  \end{tabular}}

\end{document}